\documentclass[11pt]{amsart}
\usepackage[pdftex]{hyperref}
\usepackage[all]{xy}
\usepackage{amssymb}

 \newtheorem{thm}{Theorem}[section]
 \newtheorem{cor}[thm]{Corollary}
 
 \newtheorem{prop}[thm]{Proposition}
 \theoremstyle{definition}
 
 \theoremstyle{remark}
 \newtheorem{rem}[thm]{Remark}
 
 \numberwithin{equation}{section}

\def\nn{^{(n)}}		\def\hh{\bar{h}}
\def\xx{\bar{x}}	\def\yy{\bar{y}}
\def\zz{\bar{z}}	\def\XX{\times}
	\def\mP{\mathbb P}
\def\cir{\mathop{\stackrel{\scriptscriptstyle\circ}{}}}
\def\gl{\mathop\mathrm{GL}\nolimits}
\def\th{\theta}		\def\ga{\gamma}
\def\mF{\mathbb F}	\def\si{\sigma}
\def\dd{\partial}	\def\8{\infty}
\def\ol{\overline}	\def\+{\oplus}
\def\mZ{\mathbb Z}	\def\aK{\Bbbk}
\def\la{\lambda}	\def\mN{\mathbb N}
\def\gG{\mathfrak g}	\def\al{\alpha}
	\def\be{\beta}
\def\bA{\mathbf{A}}		\def\2{^{(2)}}	
\def\dA{\mathfrak{A}}	\def\mzg{\mathbb{Z}_{\ge0}}
\def\ka{\kappa}		\def\dF{\mathfrak{F}}
\def\tP{\tilde{\mP}}	\def\eps{\varepsilon}
\def\tF{\tilde{\dF}}	\def\tA{\tilde{\dA}}
\def\dR{\mathfrak{R}}

\def\gnr#1{\langle\,#1\,\rangle}
\def\gnrsuch#1#2{\langle\,#1\mid #2\,\rangle}
\def\lst#1#2{ #1_1 , #1_2 , \dots , #1_{#2} }
\def\row#1#2{( #1_1 , #1_2 , \dots , #1_{#2} )}
\def\bop{\bigoplus}
\def\mtr#1{\begin{pmatrix}#1\end{pmatrix}}
\def\smtr#1{\left(\begin{smallmatrix}#1\end{smallmatrix}\right)}
\def\chr{\mathop\mathrm{char}}

\def\FF{\phantom{\text{\LARGE I}}}

\def\chg{Cher\-ni\-kov $p$-group}
\def\iff{if and only if }

\begin{document}

%
%
%
%
%
%
%

\title[On nilpotent Chernikov $p$-groups]
 {On nilpotent Chernikov $p$-groups\\ with elementary tops}
\author{Yuriy Drozd\and Andriana Plakosh}
\address{Institute of Mathematics, National Academy of Sciences of Ukraine, and
Max-Plank-Institut f\"ur Mathematik, Bonn}
\email{drozd@imath.kiev.ua,\,y.a.drozd@gmail.com}
\email{andrianalomaga@mail.ru}
\urladdr{http://www.imath.kiev.ua/$\sim$drozd}




\begin{abstract}
 The description of nilpotent \chg{s} with elementary tops is reduced to the study of tuples of skew-symmetric
 bilinear forms over the residue field $\mF_p$. If $p\ne2$ and the bottom of the group only consists of $2$ 
 quasi-cyclic summands, a complete classification is given. The main tool is the theory of representations of 
 quivers with involution.
\end{abstract}

\maketitle
\tableofcontents

\section{Structure theorem}
\label{s1} 

 Recall that a \emph{\chg} \cite{ch,kur} $G$ is an extension of a finite direct sum $M$ of \emph{quasi-cyclic} $p$-groups, 
 or, the same, the groups of type $p^\8$, by a finite $p$-group $H$. Note that $M$ is the biggest abelian 
 divisible subgroup of $G$, so both $M$ and $H$ are defined by $G$ up to isomorphism. We call $H$ and $M$,
 respectively, the \emph{top} and the \emph{bottom} of $G$. We denote by $M\nn$ a direct sum
 of $n$ copies $M_i$ of quasi-cyclic $p$-groups and fix elements $a_i\in M_i$ of order $p$. A Chernikov $p$-group
 is defined by an action of a finite $p$-group $H$ on a group $M\nn$ and an element from the second cohomology group 
 $H^2(H,M\nn)$ with respect to this action. Such an element is given by a $2$-cocycle $\mu:H\XX H\to M\nn$, which is
 defined up to a $2$-boundary \cite[Chapter~15]{hall}. In what follows it is convenient to denote the operations
 in the groups $G,H,M$ by $+$, so their units are denoted by $0$.
 
 It is known \cite[Theorem~1.9]{ch} that a \chg\ $G$ is nilpotent \iff the action of $H$ on $M\nn$ is trivial. In this case
 a cocycle is a map $\mu:H\XX H\to M\nn$ such that $\mu(y,z)+\mu(x,y+z)=\mu(x+y,z)+\mu(x,y)$ for all $x,y,z\in H$.
 We can also suppose that $\mu$ is \emph{normalized}, i.e. $\mu(0,x)=\mu(x,0)=0$ for every $x\in H$. A coboundary
 of a function $\ga:H\to M\nn$ is the function $\dd\ga(x,y)=\ga(x)+\ga(y)-\ga(x+y)$.
 
 Let $H_m$ be the elementary abelian $p$-group with $m$ generators, 
 \[
  H_m=\gnrsuch{\lst hm}{ph_i=0,\,h_i+h_j=h_j+h_i\ \text{for all } i,j}.
 \]
 Let also $M_p\nn=\{a\in M\nn\,|\,pa=0\}=\gnr{\lst an}$. We denote by $S(n,m)$ the group of all skew-symmetric
 maps $\tau:H_m\XX H_m\to M_p\nn$, i.e. such bilinear maps that $\tau(x,x)=0$ for all $x$ (hence 
 $\tau(x,y)=-\tau(y,x)$ for all $x,y$).
 
 \begin{thm}\label{11} 
 \emph{(cf. \cite{sha})}\ 
 If $H_m$ acts trivially on $M\nn$, then $H^2(H_m,M\nn)\simeq S(n,m)$.
 \end{thm}
  \begin{proof}
  Let $G$ be an extension of $M\nn$ by $H_m$ with the trivial action of $H_m$ corresponding to a cocycle $\mu$.
  Then for every $x\in H_m$ there is a representative $\xx\in G$ such that $\xx+\yy=\ol{x+y}+\mu(x,y)$.
  Set
  \[
    t(x,y)=[\xx,\yy]=(\ol{x+y}+\mu(x,y))-(\ol{y+x}+\mu(y,x))=\mu(x,y)-\mu(y,x),
  \]  
  since all values $\mu(x,y)$ are in the center of $G$. As all commutators are in the center of $G$ as well, we have
 \begin{align*}
   [\ol{x+y},\zz]&=(\xx+\yy-\mu(x,y)+\zz)-(\zz+\xx+\yy-\mu(x,y))\\&=(\xx+\yy+\zz)-(\zz+\xx+\yy)\\
   		&=\xx+\yy+\zz-\yy-\xx-\zz\\&=\xx+\yy+\zz-\yy-\zz+\zz-\xx-\zz\\&=\xx+[\yy,\zz]+\zz-\xx-\zz\\
   		&=[\xx,\zz]+[\yy,\zz].
 \end{align*}  
 Thus the function $t:H_m\XX H_m\to M\nn$ is bilinear. Obviously, it is skew-symmetric. Moreover, $pt(x,y)=t(px,y)=t(0,y)=0$,
 so $t(x,y)\in M_p\nn$. We denote this function by $\tau(\mu)$, so defining a map
 $\tau:Z^2(H_m,M\nn)\to S(n,m)$, where $Z^2$ denotes the group of cocycles. 
 
 If $\mu=\dd\ga$, it is symmetric: $\mu(x,y)=\mu(y,x)$, hence $\tau(\mu)=0$. On the contrary, let $\tau(\mu)=0$.
 Then the group $G$ is commutative. Therefore, its divisible subgroup $M\nn$ is a direct summand of $G$ 
 \cite[Theorem~13.3.1]{hall}, i.e. $G=M\nn\+H_m$, so the class of $\mu$ in $H^2(H_m,M\nn)$ is zero. It means that
 $\mu$ is a coboundary. Thus $\ker\tau=B^2(H_m,M\nn)$, the group of coboundaries. 
 
 It remains to prove that $\tau$ is surjective. Let $t:H_m\XX H_m\to M_p\nn$ be any skew-symmetric function. Set
 $t_{ij}=t(h_i,h_j)$ and, for any elements $x=\sum_{i=1}^m\al_ih_i$, $y=\sum_{j=1}^m\be_jh_j$, set
 $\mu(x,y)=\sum_{i<j}\al_i\be_jt_{ij}$. If $z=\sum_{k=1}^m\ga_kh_k$, then
 \begin{align*}
  \mu(y,z)+\mu(x,y+z)&=\sum_{i<j}\be_i\ga_jt_{ij}+\sum_{i<j}\al_i(\be_i+\ga_j)t_{ij},\\
  \mu(x+y,z)+\mu(x,y)&=\sum_{i<j}(\al_i+\be_i)\ga_jt_{ij}+\sum_{i<j}\al_i\be_it_{ij},
 \end{align*}
 so both sums equal $\sum_{i<j}(\al_i\be_j+\al_i\ga_j+\be_i\ga_j)t_{ij}$. Hence $\mu$ is a cocycle. Moreover, 
\[
  \mu(h_i,h_j)-\mu(h_j,h_i)=\begin{cases}
  t_{ij} &\text{if } i<j,\\
  -t_{ji}=t_{ij} &\text{if }i>j,
  \end{cases}
 \] 
 whence $\tau(\mu)=t$.
  \end{proof}
    
 Now we can classify all nilpotent \chg{s} which are extensions of $M\nn$ by $H_m$ up to isomorphism. As we have seen,
 such a group is generated by the subgroup $M\nn$ and elements $\row{\hh}m$ with the defining relations
 \begin{align*}\label{e1} 
&  \hh_i+a=a+\hh_i,\\
&  p\hh_i=0,\\
&  [\hh_i,\hh_j]=t_{ij}
 \end{align*}
 for all $a\in M\nn$ and all $i,j\in\{1,2,\dots,m\}$,
 where $(t_{ij})$ is a skew-symmetric $m\XX m$ matrix with elements from $M_p\nn$. As $M_p\nn\simeq\mF_p^n$, where $\mF_p$
 is the residue field modulo $p$, the matrix $(t_{ij})$ can be considered as an $n$-tuple $\bA=\row An$ of $m\XX m$ 
 skew-symmetric matrices with elements from $\mF_p$. Recall that both $M\nn$ and $H_m$ are uniquely defined by $G$.
 
 \begin{thm}\label{12} \emph{(cf. \cite{gu})}\ 
 Let $G$ and $F$ be two nilpotent \chg{s} with tops $H_m$ and bottoms $M\nn$\!, $\,t$ and $f$ be the corresponding 
 skew-symmetric functions $H_m\XX H_m\to M_p\nn$. The groups $G$ and $F$ are isomorphic \iff there are automorphisms 
 $\si$ of $M\nn$ and $\th$ of $H_m$ such that $f(\th(x),\th(y))=\si(t(x,y))$ for all $x,y\in H_m$.
 \end{thm}
 \begin{proof}
  As $M\nn$ is the biggest divisible abelian subgroup of $G$ or $F$, any isomorphism $\phi:G\to F$ maps $M\nn$ to itself, so
  defines automorphisms $\si=\phi|_{M\nn}$ of $M\nn$ and $\th$ of $H_m=G/M\nn=F/M\nn$. Note that the functions $t$ and $f$ 
  do not depend on the choice of representatives of elements from $H$ in $G$ and $F$. If $\xx$ is a preimage of $x\in H_m$ 
  in $G$, then $\xx'=\phi(\xx)$ is a preimage of $\th(x)$ in $F$. Therefore, 
  $f(\th(x),\th(y))=[\xx',\yy']=\phi([\xx,\yy])=\si(t(x,y))$.
  
  On the other hand, if $\si$ and $\th$ are automorphisms satisfying the condition of the theorem, the map $\phi:G\to F$ 
  such that $\phi(a)=\si(a)$ for $a\in M\nn$ and $\phi(\xx)=\ol{\th(x)}$ defines an isomorphism between $G$ and $F$.
 \end{proof}
  
 If we identify skew-symmetric functions with $n$-tuples of skew-symmetric matrices over the field $\mF_p$, this theorem
 can be reformulated as follows. For any $n$-tuple $\bA=\row An$ and any invertible matrix $Q=(q_{ij})\in\gl(n,\mF)$ we
 set $\bA\cir Q=\row{A'}n$, where $A'_j=\sum_{i=1}^nq_{ij}A_i$. If the matrices $A_i$ are of size $m\XX m$ and 
 $P\in\gl(m,\mF)$, we set $$P\cir\bA=(PA_1P^\top,PA_2P^\top,\dots,PA_nP^\top),$$ where $P^\top$ denotes the transposed matrix.
  Obviously, these two operations commute. The $n$-tuples $\bA$ and $P\cir \bA$ are said to be \emph{congruent}, and the
  $n$-tuples $\bA$ and $P\cir\bA\cir Q$ are called \emph{weakly congruent}.
  
 \begin{cor}\label{13} 
  Two $n$-tuples $\bA$ and $\bA'$ of $m\XX m$ skew-symmetric matrices over $\mF_p$ define isomorphic nilpotent \chg{s}
  with tops $H_m$ and bottoms $M\nn$ \iff they are weakly congruent.
 \end{cor}
 \begin{proof}
  The transformation $A\mapsto P\cir \bA$ corresponds to an automorphism of $H_m\simeq\mF_p^m$ given by the matrix $P$. On the 
  other hand, automorphisms of $M\nn$ are given by invertible matrices $Q$ from $\gl(n,\mZ_p)$, where $\mZ_p$ is the ring of
  $p$-adic integers considered as the endomorphism ring of the group of type $p^\8$ \cite[\S\,21]{kur}. Such an
  automorphism transforms a sequence of matrices $\bA$ to $\bA\cir Q$. Moreover, the result only depends on the value of $Q$
  modulo $p$. As every invertible matrix over $\mF_p$ can be lifted to an invertible matrix over $\mZ_p$, it accomplishes
  the proof.
 \end{proof}
 
 We denote by $G(\bA)$ the nilpotent \chg\ with the bottom $M\nn$ and elementary top corresponding to an $n$-tuple
 of skew-symmetric matrices $\bA$.
 
 \section{Relation with representations of quivers}
 \label{s2} 
 
 Theorem~\ref{12} and Corollary~\ref{13} reduce the classification of nilpotent \chg{s} with top $H_m$ and bottom $M\nn$ 
 up to isomorphism to a problem of linear algebra, namely, to the classification of $n$-tuples of skew-symmetric bilinear
 forms over the residue field $\mF_p$. If $p\ne2$, this problem is closely related with the study of representations of 
 the so called \emph{generalized Kronecker quiver} 
 \[
  K_n=\xymatrix{ 1 \ar@/^3ex/[rr]|{\,a_1\,} \ar@/^1.5ex/[rr]|{\,a_2\,} \ar@{}[rr]|{\vdots} \ar@/_2.5ex/[rr]|{\,a_n\,} && 2}.
 \]
 Recall this relation \cite{ser}. A \emph{representation} $R$ of $K_n$ over a field $\aK$ consists of two finite dimensional 
 vector spaces $R(1)$ and $R(2)$ and $n$ linear maps $R(a_i):R(1)\to R(2)\ (1\le i\le n)$. A \emph{morphism} $f$ from 
 a representation $R$ to a representation $R'$ is a pair of linear maps $f(k):R(k)\to R'(k)\ (k=1,2)$ such that 
 $f(2)R(a_i)=R'(a_i)f(1)$ for all $1\le i\le n$. We define an involution $^*$ on the quiver $K_n$ setting 
 $1^*=2,\,2^*=1$ and $a_i^*=-a_i$ for all $1\le i\le n$. If $R$ is a representation of $K_n$, we define the dual 
 representation $R^*$ setting $R^*(k)=R(k^*)^*$, where $V^*$ denotes the dual vector space to $V$, and $R^*(a_i)=-R(a_i)^*$, 
 where $L^*:W^*\to V^*$ denotes the dual linear map to $L:V\to W$. A representation $R$ is said to be \emph{self-dual} if
 $R^*=R$. Then $R(a_i):R(1)\to R(1)^*$ is identified with a bilinear form on $R(1)$ and, if $\chr\aK\ne2$, this form is
 skew-symmetric, since $R(a_i)^*=-R(a_i)$. One can check (cf. \cite{ser}) that a representation $R$ is isomorphic to a 
 self-dual one \iff there is a \emph{self-dual isomorphism} $f:R\to R^*$, i.e. such an isomorphism that $f(2)=f(1)^*$. 
 We usually identify a representation $R$ with the $n$-tuple of matrices describing the linear maps $R(a_i)$.
 
 Let $R$ be an indecomposable representation of $K_n$ which is not isomorphic to a self-dual one. Then $=R\+R^*$ is 
 isomorphic to a self-dual representation $R^+$, which cannot be decomposed into a direct sum of non-zero self-dual
 representations. Namely, $R^+$ is given by the $n$-tuple of skew-symmetric matrices
\[
  R^+(a_i)=\mtr{0&R(a_i)\\-R(a_i)^\top&0}.
 \] 
 If $\chr\aK\ne2$,
 every self-dual representation decomposes into a direct sum of indecomposable self-dual representations
 and representations of the form $R^+$, where $R$ is an indecomposable representation which is not isomorphic to any 
 self-dual one. Moreover, the direct summands of the form $R^+$ are defined uniquely up to permutation, isomorphisms 
 of the corresponding indecomposable representations $R$ and replacing $R$ by $R^*$ \cite[Theorem~1]{ser}.
 
 Obviously, if $n=1$, there are no indecomposable self-dual representations. In the next section we will see that the same 
 holds for $n=2$. On the contrary, if $n=3$, the representation $R$ such that $R(1)=R(2)=\aK^3$ and
 \[
  R(a_1)=\mtr{0&1&0\\-1&0&0\\0&0&0}\!,\ R(a_2)=\mtr{0&0&1\\0&0&0\\-1&0&0}\!,\
  R(a_3)=\mtr{0&0&0\\0&0&1\\0&-1&0}
 \]
 is indecomposable and self-dual. 
 
 Actually, a classification of representations of the quiver $K_n$ for $n>2$ is a so-called \emph{wild problem}. It means
 that it contains the classification of representations of every finitely generated algebra over the field $\aK$ (see
 \cite{dr} for precise definitions). The same is true for representations which are not isomorphic to self-dual. 
 Namely, let $n=3$, $R(1)=\aK^d,\,R(2)=\aK^{2d}$, 
 \[
  R(a_1)=\mtr{I_d\\0}\!,\ R(a_2)=\mtr{0\\I_d}\!,\ R(a_3)=\mtr{X\\Y}\!,
 \]
 where $I_d$ is the unit $d\XX d$ matrix, $X,Y$ are arbitrary square $d\XX d$ matrices. Obviously, $R$ is not self-dual. 
 One can easily check that two such representations given by the pairs $(X,Y)$ and $(X',Y')$ are isomorphic \iff the 
 pairs $(X,Y)$ and $(X',Y')$ are conjugate, i.e. $X'=SXS^{-1},\,Y'=SYS^{-1}$ for some invertible matrix $S$. 
 The problem of classification of pairs of square matrices up to conjugacy is a ``standard'' wild problem \cite{dr}. 
 Thus one cannot hope to get a more or less comprehensible classification of triples of skew-symmetric forms. This is even 
 more so for $n$-tuples with $n>3$. In the next section we will see that for $n=2$ the problem is ``\emph{tame}''\!, hence 
 there is a quite clear description of the corresponding groups.
 
 \begin{rem}\label{21} 
  If $\chr\aK=2$, the definition of a skew-symmetric bilinear form cannot be ``linearised'', since the condition $B(x,x)=0$
  is no more the consequence of the condition $B(x,y)=-B(y,x)$. Hence, we cannot identify $n$-tuples of skew-symmetric forms
  with self-dual representations of the quiver $K_n$. Moreover, the results of \cite{ser} are also valid only if 
  $\chr\aK\ne2$. Thus, to study Chernikov $2$-groups, we have to use quite different methods.
 \end{rem}
 
 \section{Case $n=2$}
 \label{s3}
 
 If $n=1$, $G$ is described by one skew-symmetric matrix $A$. This matrix is congruent to a direct sum of $k$ matrices
 $\smtr{0&1\\-1&0}$ and $l$ matrices $(0)$, where $m=2k+l$. It gives a simple description.
 
 \begin{prop}\label{30} 
  A nilpotent \chg\ $G$ with elementary top and quasi-cyclic bottom $M$ decomposes as $G_k\XX H_l$, where $G_k$
 is generated by $M$ and $2k$ elements $\lst{\hh}{2k}$ which are of order $p$, commute with all elements from $M$
 and their commutators $[\hh_i,\hh_j]$ for $i<j$ are given by the rule
 \[
 [\hh_i,\hh_j]=\begin{cases}
 a_1 &\text{\emph{if }} j=k+i,\\
 0 &\text{\emph{otherwise}},
\end{cases}
 \]
 where $a_1$ is a fixed element of order $p$ from the group $M$.
 \end{prop} 

 Now we consider the case $n=2$.
 
 Following the preceding consideration, we classify the pairs of skew-symmetric bilinear forms over a field $\aK$ with
 $\chr\aK\ne2$. Equivalently, we classify the self-dual representations of the Kronecker quiver $K_2$ with the
 involution $1^*=2,\,2^*=1,\,a_i^*=-a_i$. Recall \cite[Chapter~XII]{gan} that indecomposable representations 
 of $K_2$ (``matrix pencils'') are given by the following pairs of matrices;
 \begin{equation}\label{e31} 
 \begin{split}
  R_f:&\quad R_f(a_1)=I_d,\ R_f(a_2)=F(f),\\
  R_{\8,d}:&\quad R_{\8,d}(a_1)=F(x^d),\ R_{\8,d}(a_2)=I_d,\\
  R_{-,d}:&\quad R_{-,d}(a_1)=\mtr{1&0&0&\dots&0&0\\ 0&1&0&\dots&0&0\\\hdotsfor6\\0&0&0&\dots&1&0},\\
  		 &\quad R_{-,d}(a_2)=\mtr{0&1&0&\dots&0&0\\ 0&0&1&\dots&0&0\\\hdotsfor6\\0&0&0&\dots&0&1},\\
  R_{+,d}:&\quad R_{+,d}(a_i)=R_{-,d}(a_i)^\top.  
 \end{split}
\end{equation}  
 Here $f=f(x)$ is a polynomial of degree $d$ from $\aK[x]$ which is a power of a unital irreducible polynomial and 
 $F(f)$ is the Frobenius matrix with the characteristic polynomial $f(x)$.
 The size of the matrices in $R_{-,d}$ is $(d-1)\XX d$; respectively, the size of the matrices
 in $R_{+,d}$ is $d\XX(d-1)$. Actually, $R_{+,d}=(R_{-,d})^*$, $R_f^*\simeq R_f$ and 
 $R_{\8,d}^*\simeq R_{\8,d}$. Nevertheless, there are no self-dual indecomposable representations.
 
 \begin{prop}\label{31} 
  Neither of indecomposable representations from the preceding list is isomorphic to a self-dual one.
 \end{prop}
 \begin{proof}
  It is evident for $R_{\pm,d}$. The representation $R_f^*$ is given by the pair of matrices $(-I_d,-F(f)^\top)$. If it were
  isomorphic to a self-dual one, there would be an invertible $d\XX d$ matrix $P$ such that $PI_d=-I_dP^*$ and 
  $PF(f)=-F(f)^\top P^*$. Hence $P$ is skew-symmetric, and $PF(f)=F(f)^\top P$. One easily checks that 
  it is impossible. The same holds for $R_{\8,d}$.
 \end{proof}
 
 Combining this result with those from \cite{ser}, we get a complete classification of pairs of skew-symmetric 
 bilinear forms. We denote by $\dA$ the set of all pairs $R^+$, where $R\in\{R_f,R_{\8,d},R_{-,d}\}$, and by $\dF$
 the set of functions $\ka:\dA\to\mzg$ such that $\ka(\bA)=0$ for almost all $\bA$. For any function
 $\ka\in\dF$ we set $\dA^\ka=\bop_{\bA\in\dA}\bA^{\ka(\bA)}.$
 
 \begin{thm}\label{32} 
 Let $\chr\aK\ne2$.
 Any pair of skew-symmetric bilinear forms over the field $\aK$ is congruent to a direct sum $\dA^\ka$ for a uniquely 
 defined function $\ka\in\dF$.
 \end{thm} 

 To obtain a classification of \chg{s} with elementary tops and the bottom $M\2$, we also have to answer the question:
 
  \smallskip
 \emph{Given two functions  with finite supports $\ka,\ka':\dA\to\mzg$, when are the pairs $\dA^\ka$ and $\dA^{\ka'}$
 weakly congruent}?
 
 \smallskip
 Evidently, $(\bA_1\+\bA_2)\cir Q=(\bA_1\cir Q)\+(\bA_2\cir Q)$, so the pairs $\bA$ and $\bA\cir Q$ are indecomposable 
 simultaneously. For every pair $\bA\in\dA$ we denote by $\bA*Q$ the unique pair from $\dA$ which is congruent to $\bA\cir Q$.
 The map $\bA\mapsto\bA*Q$ defines an action of the group $\gG=\gl(2,\aK)$ on the set $\dA$, hence on the set $\dF$ of functions
 $\ka:\dA\to\mzg$: $(Q*\ka)(\bA)=\ka(A*Q)$. Theorem~\ref{32} implies the following result.
 
 \begin{cor}\label{33} 
  The pairs $\dA^\ka$ and $\dA^{\ka'}$ are weakly congruent \iff the functions $\ka$ and $\ka'$ belong to the same 
  orbit of the group $\gG$.
 \end{cor}
 
 Obviously, $R^+\cir Q=(R\cir Q)^+$ for every representation $R$ of the quiver $K_2$. Thus we have to know when 
 $R\cir Q\simeq R'$ for indecomposable representations from the list \eqref{e31}. As $R_{-,d}$ is a unique (up to
 isomorphism) indecomposable representation $R$ such that $\dim R(1)=d-1,\,\dim R(2)=d$, we only have to consider 
 the representations from the set $\{R_f,R_{\8,d}\}$. From \cite[Chapter~XII,\ \S\,3]{gan} it follows that a pair 
 $R=(R_1,R_2)$ from this set is completely defined by its \emph{homogeneous characteristic polynomial} 
 $\chi_R(x_1,x_2)=\det(x_1R_1-x_2R_2)$. Actually, $\chi_{R_f}=x_2^df(x_1/x_2)$, where $d=\deg f$, and 
 $\chi_{R_{\8,d}}=x_2^d$. The group $\gG$ naturally acts on the ring $\aK[x_1,x_2]$: 
 $Q\cir f=f(q_{11}x_1+q_{12}x_2,q_{21}x_1+q_{22}x_2)$, where $Q=(q_{ij})$, and
 \begin{align*}
    \chi_{R\cir Q}&=\det\big((q_{11}R_1+q_{21}R_2)x+(q_{12}R_1+q_{22}R_2)\big)=\\
    			&=\det\big((q_{11}x+q_{12})R_1+(q_{21}x+q_{22})R_2)=
    			Q\cir\chi_R.
 \end{align*}
 We say that an irreducible homogeneous polynomial $g\in\aK[x_1,x_2]$ is \emph{unital} if either $g=x_2$ or its leading 
 coefficient with respect to $x_1$ equals $1$. Let $\mP=\mP(\aK)$ be the set of unital homogeneous irreducible polynomials 
 from $\aK[x_1,x_2]$ and $\tP=\tP(\aK)=\mP\cup\{\eps\}$. Note that $\mP$ actually coincides with the set of the closed 
 points of the projective line $\mP^1_{\aK}=\mathrm{Proj}\,\aK[x_1,x_2]$ \cite{ha}. For $g\in\mP$ and $Q\in \gG$, 
 let $Q*g$ be the unique polynomial $g'\in\mP$ such that $Q\cir g=\la g'$ for some non-zero $\la\in\aK$. 
 (It is the natural action of $\gG$ on $\mP^1_{\aK}$.) We also set $Q*\eps=\eps$ for any $Q$. It defines an action 
 of $\gG$ on $\tP$. Denote by $\tF=\tF(\aK)$ the set of all functions $\rho:\tP\XX\mN\to\mzg$ such that $\rho(g,d)=0$ 
 for almost all pairs $(g,d)$. Define the actions of the group $\gG$ on $\tF$ setting $(\rho*Q)(g,d)=\rho(Q*g,d)$. 
 For every pair $(g,d)\in\tF$ we define a pair of skew-symmetric
 forms $R(g,d)$:
 \[
  R(g,d)=\begin{cases}
  R_{-,d}^+ &\text{if } g=\eps,\\
  R_{\8,d}^+ &\text{if } g=x_2,\\
  R_{g(x,1)^d}^+ &\text{otherwise}.
  \end{cases}
 \]
 Let $\tA=\tA(\aK)=\{R(g,d)\,|\,(g,d)\in\tP\XX\mN\}$. For every function $\rho\in\tF$ we set
 $\tA^\rho=\bop_{(g,d)\in\tP\XX\mN}R(g,d)^{\rho(g,d)}$.
 The preceding considerations imply the following theorem.

 \begin{thm}\label{34} 
 Let $\chr\aK\ne2$.
 \begin{enumerate}
 \item  Every pair of skew-symmetric bilinear forms over the field $\aK$ is weakly congruent to $\tA^\rho$ for some function 
 $\rho\in\tF(\aK)$.
 \item  The pairs $\tA^\rho$ and $\tA^{\rho'}$ are weakly congruent \iff the functions $\rho$ and $\rho'$ belong to
 the same orbit of the group $\gG=\gl(2,\aK)$.
 \end{enumerate}
 \end{thm}
  
 From Theorem~\ref{34} and Corollary~\ref{13} we immediately obtain a classification of nilpotent \chg{s} with elementary 
 tops and the bottom $M\2$. Namely, for every function $\rho\in\tF(\mF_p)$ set $G(\rho)=G(\tA^\rho)$.
 
 \begin{thm}\label{35} 
 Let $\dR$ be a set of representatives of orbits of the group $\gG=\gl(2,\mF_p)$ acting on the set of functions
 $\tF(\mF_p)$. Then every nilpotent \chg\ with elementary top and the bottom $M\2$ is isomorphic to the group
 $G(\rho)$ for a uniquely defined function $\rho\in\dR$.
 \end{thm}
  
 One can easily describe these groups in terms of generators and relations. Note that all of them are of the form
 $G(\bA)$, where $\bA=\bop_{k=1}^s\bA_k$ and all $\bA_k$ belong to the set $\{R^+_{-,d},R^+_{\8,d},R^+_f\}$. Therefore
 $G(\bA)$ is generated by the subgroup $M\2$ and elements $\hh_{ki}$, where $1\le k\le s,\,1\le i\le d_k$, 
 $d_k=2\deg f$ if $\bA_k=R^+_f$, $d_k=2d$ if $\bA_k=R^+_{\8,d}$, and $d_k=2d-1$ if $\bA=R^+_{-,d}$. 
 All elements $\hh_{ki}$ are of order $p$, 
 commute with the elements from $M\2$, $[\hh_{ki},\hh_{lj}]=0$ if $k\ne l$ and the values of the commutators 
 $[\hh_{ki},\hh_{kj}]$ for $i<j$, according to the type of $\bA_k$, are given in Table~1. In this table $a_1$ 
 and $a_2$ denote some fixed generators of the subgroup $M_p\2$.
 
 \begin{table}[!ht]
\caption{}
 \[
  \begin{array}{|c|c|c|}
  \hline
  \FF \bA_k & i,j & [\hh_{ki},\hh_{kj}] \\
  \hline
    \FF R^+_{-,d}\ & j=d+i & a_1\\
  & j=d+i-1 & a_2 \\
  & \text{ otherwise } & 0  \\
  \hline
  \FF R^+_{\8,d} & j=d+i & a_2,\\
  & j=d+i-1 & a_1,\\
  & \text{ otherwise } & 0 \\
  \hline
  \FF R^+_f &\ j=d+i<2d\ & a_1\\
   &j=d+i-1 &
  a_2 \\
  & i<d,\,j=2d &\  -\la_{d-i+1}a_2\ \\
  & i=d,\,j=2d &\ a_1-\la_1a_2\  \\
  & \text{ otherwise } & 0  \\
  \hline
  \end{array}
 \]
 \centerline{ where $f(x)=x^d+\la_1 x^{d-1}+\dots+\la_d$}
 \end{table}

 \begin{cor}\label{36} 
 Let $G=G(\bA)$. 
 \begin{enumerate}
 \item  $G$ has a finite direct factor \iff $\bA\simeq(R_{-,1})^k\+\bA'$; then $G\simeq H_k\XX G(\bA')$.
 
 \item  Suppose that $G$ has no finite direct factors. It is decomposable \iff $\bA\simeq(R^+_x)^{k}\+(R^+_{\8,1})^{l}$;
 then $G=G_k\XX G_l$.
 \end{enumerate}
 \emph{(See Proposition~\ref{30} for the definition of $G_k$.)}
 \end{cor}
 \begin{proof}[Proof \emph{is evident}]
 \end{proof}


\subsection*{Acknowledgment}

  The authors are grateful to Igor Shapochka who initiated this investigation.

\end{document}